\documentclass{amsart}
\usepackage{amssymb}
\usepackage{stmaryrd} 
\usepackage{amsmath} 
\usepackage{amscd}
\usepackage{amsbsy}
\usepackage{commath}
\usepackage{comment, enumerate}
\usepackage[matrix,arrow]{xy}
\usepackage{hyperref}
\usepackage{mathrsfs}
\usepackage{color}
\usepackage{float}
\usepackage{mathtools,caption}
\usepackage{tikz-cd}
\usepackage{longtable}
\usepackage[utf8]{inputenc}
\usepackage[OT2,T1]{fontenc}
\usepackage[normalem]{ulem}
\usepackage{hyperref}
\usepackage[margin=1.2 5in]{geometry}
\hypersetup{
 colorlinks=true,
 linkcolor=blue,
 filecolor=violet,
 citecolor=orange,
 urlcolor=purple,
 pdftitle={Anticyclotomic p not q},
 }

\DeclareSymbolFont{cyrletters}{OT2}{wncyr}{m}{n}
\DeclareMathSymbol{\Sha}{\mathalpha}{cyrletters}{"58}

\definecolor{Green}{rgb}{0.0, 0.5, 0.0}

\newcommand{\Qd}[1]{\QQ(\sqrt{-{#1}})}
\DeclareMathOperator{\Hom}{Hom}
\DeclareMathOperator{\Gal}{Gal}
\DeclareMathOperator{\alg}{alg}
\DeclareMathOperator{\cyc}{cyc}
\DeclareMathOperator{\ord}{ord}

\DeclareMathOperator{\Cl}{Cl}

\DeclareMathOperator{\coker}{coker}
\DeclareMathOperator{\Sel}{Sel}

\DeclareMathOperator{\charac}{char}
\DeclareMathOperator{\lcm}{lcm}

\newcommand{\QQ}{\mathbb Q}
\newcommand{\ZZ}{\mathbb Z}
\newcommand{\CC}{\mathbb C}
\newcommand{\Zp}{\mathbb{Z}_p}
\newcommand{\Qp}{\mathbb{Q}_p}
\newcommand{\Zq}{\mathbb{Z}_q}

\newcommand{\cF}{\mathcal F}

\newcommand{\cO}{\mathcal O}
\newcommand{\fa}{\mathfrak a}
\newcommand{\ff}{\mathfrak f}
\newcommand{\fb}{\mathfrak b}

\newcommand{\fp}{\mathfrak p}
\newcommand{\p}{\mathfrak p}
\newcommand{\fq}{\mathfrak q}
\newcommand{\fP}{\mathfrak P}
\newcommand{\fQ}{\mathfrak Q}
\newcommand{\fg}{\mathfrak g}
\newcommand{\fr}{\mathfrak r}
\newcommand{\fh}{\mathfrak h}

\newcommand{\cC}{\mathcal{C}}
\newcommand{\cZ}{\mathcal{Z}}

\newcommand{\sR}{\mathscr{R}}
\newcommand{\cG}{\mathcal{G}}

\newcommand{\MQp}{\mathfrak{m}_{\overline{\Qp}}}
\newcommand{\q}{\mathfrak{q}}
\newcommand{\ac}{{\mathrm{ac}}}

\newtheorem*{Theorem*}{Theorem}
\newtheorem{Th}{Theorem}[section]
\newtheorem{Lemma}[Th]{Lemma}
\newtheorem*{Ques*}{Question}

\newtheorem{Cor}[Th]{Corollary}

\newtheorem*{conj*}{Conjecture}

\newtheorem{lthm}{Theorem}

\theoremstyle{definition}
\newtheorem{Defi}[Th]{Definition}
\theoremstyle{remark}
\newtheorem{rem}[Th]{Remark}

\begin{document}
\title[$p\neq q$ Iwasawa theory of class groups and fine Selmer groups]{Growth of $p$-parts of ideal class groups and fine Selmer groups in $\mathbb{Z}_q$-extensions with $p\neq q$}

\author[D.~Kundu]{Debanjana Kundu}
\address[Kundu]{Fields Institute \\ University of Toronto \\
 Toronto ON, M5T 3J1, Canada}
\email{dkundu@math.toronto.edu}

\author[A.~Lei]{Antonio Lei}
\address[Lei]{Department of Mathematics and Statistics\\University of Ottawa\\
150 Louis-Pasteur Pvt\\
Ottawa, ON\\
Canada K1N 6N5}
\email{antonio.lei@uottawa.ca}

\date{\today}


\begin{abstract}
Fix two distinct odd primes $p$ and $q$.
We study "$p\ne q$" Iwasawa theory in two different settings.\newline
(1) Let $K$ be an imaginary quadratic field of class number 1 such that both $p$ and $q$ split in $K$.
We show that under appropriate hypotheses, the $p$-part of the ideal class groups is bounded over finite subextensions of an anticyclotomic $\mathbb{Z}_q$-extension of $K$.\newline
(2) Let $F$ be a number field and let $A_{/F}$ be an abelian variety with $A[p]\subseteq A(F)$.
We give sufficient conditions for the $p$-part of the fine Selmer groups of $A$ over finite subextensions of a $\mathbb{Z}_q$-extension of $F$ to stabilize.
\end{abstract}

\maketitle

\section{Introduction}
Let $F/\QQ$ be an algebraic number field and $F_\infty/F$ be a Galois extension with Galois group isomorphic to the additive group $\Zq$ of $q$-adic integers.
For each integer $n\geq 0$, there is a unique subfield $F_n/F$ of degree $q^n$.
Let $h(F_n)$ be the class number of $F_n$.
K.~Iwasawa showed that if $q^{e_n}$ is the highest power of $q$ dividing $h(F_n)$, then there exist integers $\lambda, \mu,\nu$ independent of $n$, such that $e_n = \mu q^n + \lambda n + \nu$ for $n\gg 0$.
On the other hand, in \cite{Was75, Was78}, L.~C.~Washington proved that for distinct primes $p$ and $q$, the $p$-part of the class number stabilizes in the \emph{cyclotomic} $\ZZ_q$-extension of an abelian number field.
Washington's results have been extended to other $\Zq$-extensions where primes are finitely decomposed.
In particular, J.~Lamplugh proved the following in \cite{Lam15}: if $p,q$ are distinct primes $\geq 5$ that split in an imaginary quadratic field $K$ of class number 1 and $F/K$ is a prime-to-$p$ abelian extension which is also unramified at $p$, then the $p$-class group stabilizes in the $\Zq$-extension of $F$ which is unramified outside precisely one of the primes above $q$.
There have also been speculations by J.~Coates on the size of the whole class group in a cyclotomic tower; see \cite{coates12}, especially the discussion in \S3 and Conjecture D.

Let $p$ and $q$ be two distinct odd primes and $K$ an imaginary quadratic field of class number 1 in which both $p$ and $q$ split.
We write $p\cO_K=\p\overline{\p}$ and $q\cO_K=\q\overline{\q}$.
Given an ideal $\fh$ of $\cO_K$, we write $\sR(\fh)$ for the ray class field of $K$ of conductor $\fh$.
In the first half of this article, we study the growth of the $p$-part of the ideal class group in a $\Zq$-anticyclotomic tower.
This generalizes \cite[Theorem~1.3]{Lam15}, where the stability of the $p$-part of the class numbers $\sR(\fg \fq^n)$ is studied.
More precisely, we prove the following result.

\begin{lthm}
\label{thmA}
Let $K$ be an imaginary quadratic field of class number 1.
Let $p$ and $q$ be distinct primes ($\geq 5$) which split in $K$.
Let $\fr$ be a fixed ideal of $\cO_K$ coprime to $pq$ such that $\fr$ is a product of split primes\footnote{In this article, a split prime of $K$ refers to a prime ideal of $\cO_K$ that lies above a rational prime that splits in $K$.}.
Let $\mathcal{F}= \sR(\fr q)$.
We assume that $p\nmid[\mathcal{F}:K]$.
Let $\sR(\fr q^\infty)^{\ac}/\mathcal{F}$ denote the anticyclotomic $\Zq$-extension and write $\cF_n$ for the unique subextension of $\sR(\fr q^\infty)^{\ac}/\mathcal{F}$ whose degree is $q^n$.
Then there exists an integer $N$ such that for all $n\ge N$,
\[
\ord_p(h(\mathcal{F}_n))=\ord_p(h(\mathcal{F}_N)).
\]
\end{lthm}

The hypothesis on $\fr$ being a product of split primes is crucial for the use of a theorem of H.~Hida, which guarantees the non-vanishing modulo $p$ of the algebraic $L$-values of anticyclotomic characters factoring through $\sR(\fr q^\infty)^\ac$ (see Theorem~\ref{thm:hida}).
To prove Theorem~\ref{thmA}, we link this non-vanishing to the stabilization of the $p$-class groups via the ($p$-adic) Iwasawa main conjecture proved by K.~Rubin \cite{Rub91}.
Our strategy is inspired by the work of Lamplugh \cite{Lam15}, which we outline below.

In \S\ref{S:CM}, we introduce an auxiliary elliptic curve $E_{/K}$ with CM by $\cO_K$ such that the conductor $\ff$ of its Hecke character is a product of split primes in $K$ with $p\nmid [\sR(\ff):K]$.
Let $\fg=\lcm(\ff,\fr)$.
By a result of Lamplugh, when the algebraic $L$-value of certain Hecke character is nonzero modulo $p$, the corresponding modules of local $p$-adic units and elliptic units over an extension generated by $E[\fp^\infty]$ coincide after taking appropriate isotypic components (see Theorem~\ref{th:L-values-trivial} for the precise statement).
Combining this with Hida's theorem, we prove in Theorem~\ref{key result for class group} that the $p$-primary Galois modules featured in the Iwasawa main conjecture stabilize in the anticyclotomic $\Zq$-extension $\sR(\fg q^\infty)^\ac/\sR(\fg q)$.
This can be translated into a statement on $p$-class groups, proving a special case of Theorem~\ref{thmA}, where the ideal $\fr$ is divisible by $\ff$ (see Theorem~\ref{auxiliary}).
To complete the proof of Theorem~\ref{thmA}, we bound the $p$-class groups over the tower $\sR(\fr q^\infty)^\ac/\sR(\fr q)$ by those over $\sR(\fg q^\infty)^\ac/\sR(\fg q)$.

In the second half of the article, we prove a general statement (see Theorem \ref{class group and fine Selmer group}) which shows that in certain $\Zq$-extensions of a number field $F$, the growth of the $p$-part of the class group is closely related to that of the $p$-primary \emph{fine} Selmer group of an abelian variety $A_{/F}$.
This subgroup of the classical $p$-primary Selmer group is denoted by $\Sel_0(A/F)$, and is obtained by imposing stronger vanishing conditions at primes above $p$ (the precise definition is reviewed in \S\ref{sec:notation}).
The following result is an application of the aforementioned theorem to the growth of the $p$-part of fine Selmer group of a fixed abelian variety $A$ over a $\Zq$-tower (which is not necessarily anticyclotomic).
\begin{lthm}
\label{thmB}
Let $p$ and $q$ be distinct odd primes.
Let $F$ be any number field and $A_{/F}$ be an abelian variety such that $A[p]\subseteq A(F)$.
Let $F_\infty/F$ be a $\Zq$-extension where the primes above $q$ and the primes of bad reduction of $A$ are finitely decomposed.
If there exists $N \ge 0$ such that for all $n\geq N$,
\[
\ord_{p}(h(F_n)) = \ord_{p}(h(F_N)),
\]
then there exists an integer $N'\ge N$ such that for all $n\ge N'$, there is an isomorphism
\[
\Sel_0(A/F_n) \simeq\Sel_0(A/F_{N'}).
\]
\end{lthm}

In particular, Theorem~\ref{thmB} applies to the setting studied by Washington \cite{Was75,Was78}.
Finally, we remark that unlike what we have found for fine Selmer groups in Theorem~\ref{thmB}, it has been shown by T.~Dokchitser and V.~Dokchitser that the $p$-part of the Tate--Shafarevich group of an abelian variety in a $\Zq$-tower can be unbounded; see \cite[Example~1.5]{Dok15}.

\subsection*{Acknowledgement}
We thank Ming-Lun Hsieh, Filippo A.~E.~Nuccio Mortarino Majno di Capriglio, and Lawrence Washington for answering our questions during the preparation of this article.
We are also indebted to the anonymous referees for their valuable comments and suggestions on earlier versions of the article.
DK acknowledges the support of a PIMS Postdoctoral Fellowship.
AL is supported by the NSERC Discovery Grants Program RGPIN-2020-04259 and RGPAS-2020-00096.

\section{Finding auxiliary CM elliptic curves}
\label{S:CM}
Let $K= \Qd{d}$ be an imaginary quadratic field of class number 1.
As discussed in the introduction, we shall work with an auxiliary CM elliptic curve $E_{/K}$ in order to prove Theorem~\ref{thmA}.
Recall that the imaginary quadratic fields of class number 1 are precisely the following
\[
\Qd{1}, \ \Qd{2}, \ \Qd{3}, \ \Qd{7}, \ \Qd{11}, \ \Qd{19}, \ 
\Qd{43}, \Qd{67}, \ \Qd{163}.
\]
For each choice of $K$, we shall write down an explicit elliptic curve $E_{/K}$ such that
\begin{enumerate}[(a)]
 \item $E$ has CM by $\cO_K$;
 \item If $\ff$ denotes the conductor of the Hecke character $\psi$ attached to $E$, then $\ff$ is only divisible by split primes of $K$;
 \item The rational primes dividing $[\sR(\ff):K]$ are either $2,3$ or primes that are non-split in $K$. 
\end{enumerate}
We remark that condition (c) ensures that the prime $p$ in the statement of Theorem~\ref{thmA} does not divide $[\sR(\ff):K]$.

If $E_{/K}$ is an elliptic curve with CM by an order $\cO$ in $K$, the $j$-invariant $j(E)$ is an integer in this case, so $E$ must be a twist of the base extension of an elliptic curve $A_{/\QQ}$.
For $d>3$, $A$ is uniquely determined (up to isomorphism over $K$) by the condition that it has CM by $\cO_K$ and its base change to $K$ has bad reduction at the ramified prime $\fP = (\sqrt{-d})$.
For $d=1,2,$ and $3$ there are several choices for the elliptic curve over $\QQ$ (see \cite[Remark~3.1]{CP19}).

When $d>3$, it follows from \cite[Theorem~3.3]{CP19} that if we twist $A_{/K}$ by a character corresponding to $K(\sqrt{\alpha})$ where $\alpha = \fP \fQ$ such that $\fQ$ is a prime of $K$ distinct from $\fP$ satisfying $\fQ\equiv u^2\sqrt{-d}\mod 4\cO_K$ for some $u\in \cO_K$, then the twisted elliptic curve (over $K$) has good reduction everywhere except at $\fQ$.
Therefore, for our purposes, it is enough to find such $\fQ$ that is a split prime in $K$. Indeed, we may choose $r\in\ZZ$ such that $(4r+\sqrt{-d})(4r-\sqrt{-d})=16r^2+d$ is an odd rational prime.
Such $r$ exists for all possible values of $d$.
For example, we may take $r$ to be $1$ when $d=43, 67, 163$.
Then $4r+\sqrt{-d}$ is a split prime of $K$ and $4r+\sqrt{-d}\equiv 1^2\sqrt{-d}\mod 4\cO_K$.
In particular, we may apply \cite[Theorem 3.3]{CP19} with $\fQ=(4r+\sqrt{-d})$ and $u=1$, resulting in a CM curve $E$ satisfying properties (a) and (b) above.

When $d<43$, we find $E_{/K}$ by inspection using the data available on \cite{lmfdb}.
In all our examples below, $E_{/K}$ has bad reduction at one or two split primes which are coprime to $6$.
In particular, the conductor of $E_{/K}$ is given by square of the product of the bad prime(s), whereas the conductor $\ff$ of the Hecke character $\psi$ attached to $E$ is given by the product of the bad prime(s) (see \cite[Theorem~12]{ST68}).
The ray class group $\Gal(\sR(\ff)/K)$ (and hence $[\sR(\ff):K]$) is computed using MAGMA \cite{magma}.

\begin{table}[H]
\begin{tabular}{c|c|c|c|c|c} 
 $d$ & $A_{/\QQ}$ & base change curve $A_{/K}$ & twisted curve $E_{/K}$ & bad prime(s) of $E_{/K}$ & $[\sR(\ff):K]$\\ [1 ex]
 \hline
 $1$ & \href{https://www.lmfdb.org/EllipticCurve/Q/64/a/4}{64.a4} & \href{https://www.lmfdb.org/EllipticCurve/2.0.4.1/256.1/CMa/1}{2.0.4.1-256.1-CMa1} & \href{https://www.lmfdb.org/EllipticCurve/2.0.4.1/25.1/CMa/1}{2.0.4.1-25.1-CMa1} & $2+\sqrt{-1}$&1 \\
 $2$ & \href{https://www.lmfdb.org/EllipticCurve/Q/256/d/1}{256.d1} & \href{https://www.lmfdb.org/EllipticCurve/2.0.8.1/1024.1/CMb/1}{2.0.8.1-1024.1-CMb1} & \href{https://www.lmfdb.org/EllipticCurve/2.0.8.1/9.3/CMa/1}{2.0.8.1-9.3-CMa1} & $1-\sqrt{-2}$&$1$ \\
 $3$ & \href{http://www.lmfdb.org/EllipticCurve/Q/27/a/4}{27.a4} & \href{http://www.lmfdb.org/EllipticCurve/2.0.3.1/81.1/CMa/1}{2.0.3.1-81.1-CMa1} & \href{https://www.lmfdb.org/EllipticCurve/2.0.3.1/2401.3/CMa/1}{2.0.3.1-2401.3-CMa1} & $2\pm\sqrt{-3}$ & 6 \\
 $7$ & \href{http://www.lmfdb.org/EllipticCurve/Q/49/a/4}{49.a4} & \href{http://www.lmfdb.org/EllipticCurve/2.0.7.1/49.1/CMa/1}{2.0.7.1-49.1-CMa1} & \href{http://www.lmfdb.org/EllipticCurve/2.0.7.1/1849.1/CMa/1}{2.0.7.1-1849.1-CMa1} & $6-\sqrt{-7}$& 21 \\
 $11$ & \href{http://www.lmfdb.org/EllipticCurve/Q/121/b/2}{121.b2} & \href{http://www.lmfdb.org/EllipticCurve/2.0.11.1/121.1/CMa/1}{2.0.11.1-121.1-CMa1} & \href{http://www.lmfdb.org/EllipticCurve/2.0.11.1/9.1/CMa/1}{2.0.11.1-9.1-CMa1} & $\frac{-1-\sqrt{-11}}{2}$ &1\\
 $19$ & \href{http://www.lmfdb.org/EllipticCurve/Q/361/a/2}{361.a2} & \href{http://www.lmfdb.org/EllipticCurve/2.0.19.1/361.1/CMa/1}{2.0.19.1-361.1-CMa1} & \href{http://www.lmfdb.org/EllipticCurve/2.0.19.1/49.3/CMa/1}{2.0.19.1-49.3-CMa1} & $\frac{-1+\sqrt{-19}}{2}$& 3 \\
 $43$ & \href{http://www.lmfdb.org/EllipticCurve/Q/1849/b/2}{1849.b2} & \href{http://www.lmfdb.org/EllipticCurve/2.0.43.1/1849.1/CMa/1}{2.0.43.1-1849.1-CMa1} & $\alpha = -43 + 4\sqrt{-43}$ & $4+\sqrt{-43}$& 29 \\
 $67$ & \href{http://www.lmfdb.org/EllipticCurve/Q/4489/b/2}{4489.b2} & \href{http://www.lmfdb.org/EllipticCurve/2.0.67.1/4489.1/CMa/1}{2.0.67.1-4489.1-CMa1} & $\alpha = -67 + 4\sqrt{-67}$ & $4+\sqrt{-67}$& 41 \\
 $163$ & \href{http://www.lmfdb.org/EllipticCurve/Q/26569/a/2}{26569.a2} & curve not in database & $\alpha = -163 + 4\sqrt{-163}$ & $4+\sqrt{-163}$& 89 \\
\end{tabular}
\end{table}

\section{A result of Hida on \texorpdfstring{$L$}{}-values of anticyclotomic Hecke characters}

Throughout this section and the next, $K$ is a fixed imaginary quadratic field of class number 1.
We fix an elliptic curve $E_{/K}$ with CM by $\cO_K$ as given in \S\ref{S:CM}.
Recall that $\psi$ denotes the Hecke character over $K$ with conductor $\ff$ attached to $E$.

We review a special case of a result of Hida from \cite{hida2} that will play a crucial role in our proof of Theorem~\ref{thmA}.

\begin{Defi}
\label{imprimitive L function}
Let $\fh$ be any integral ideal of $K$ and let $\epsilon$ be any Hecke character of $K$.
The $\fh$-\emph{imprimitive $L$-function} of $\epsilon$ is defined as follows
\begin{align*}
L_{\fh}(\epsilon,s) &= \prod_{\gcd(\nu, \fh)=1} \left( 1- \frac{\epsilon(\nu)}{(N\nu)^s}\right)^{-1}\\ &=\sum_{\gcd(\fa, \fh)=1} \frac{\epsilon(\fa)}{(N\fa)^s},
\end{align*}
where the product runs over \emph{prime ideals} $\nu$ of $K$ coprime to $\fh$, and sum is taken over \emph{integral ideals} $\fa$ coprime to $\fh$.
\end{Defi}

Fix an integral ideal $\fg$ of $K$ which is divisible by $\ff$, relatively prime to $pq$, and such that only split primes of $K$ divide $\fg$.
Let $F = \sR(\fg q)$ be the \emph{ray class field} of $K$ of conductor $\fg q$ and write $\Delta =\Gal(F/K)$.
Set $F_\infty=\bigcup_{n\ge1} \sR(\fg q^n)$; this is a $\ZZ_q^2$-extension of $F$.
 We fix an isomorphism 
\[
\Gal(F_\infty/K) \simeq \Gal(F/K) \times \Gal(K_\infty/K) = \Delta \times \Zq^2.
\]

Let $\epsilon$ be a character of $\Gal(F_\infty/K)$.
For our purpose, $\epsilon$ will be of the form $ \overline{\varphi \psi^{k}}$, where $\varphi$ is a finite-order character and $k$ is an integer between $1$ and $p-1$.
Denote by $L(\epsilon,s)$ the \emph{primitive Hecke $L$-function} of $\epsilon$.
Recall that the imprimitive (or partial) $L$-function differs from the primitive (or classical) $L$-function by a finite number of Euler factors.
Let $N_{K/\QQ}$ denote the norm map.
We can further define the \emph{primitive algebraic Hecke $L$-value},
\[
L^{\alg}(\overline{\varphi\psi^k})=L^{\alg}({\epsilon}) := \frac{L\left({\epsilon}, k\right)}{\Omega_{\infty}^{k} }= \frac{L\left(\overline{\varphi\psi^k}N_{K/\QQ}^{-k},0 \right)}{\Omega_{\infty}^{k} }.
\]
Here, $\Omega_\infty$ denotes a complex period for $E_{/\CC}$.
Similarly, given an integral ideal $\fh$ of $K$, we define the $\fh$-\emph{imprimitive algebraic Hecke $L$-value},
\[
L_\fh^{\alg}(\overline{\varphi\psi^k})=L_{\fh}^{\alg}({\epsilon}) := \frac{L_\fh\left({\epsilon}, k\right)}{\Omega_{\infty}^{k} }= \frac{L_\fh\left(\overline{\varphi\psi^k}N_{K/\QQ}^{-k},0 \right)}{\Omega_{\infty}^{k} }.
\]
Note that $L$ and $L_{\fh}$ differ by the omission of the Euler factors at primes dividing
$\fh$.

In what follows, we say that a Hecke character $\epsilon$ of $K$ is of \emph{infinity type $(a,b)$} if its infinity component sends $x$ to $x^a\overline{x}^b$.
Under this convention, $\psi$ has infinity type $(-1,0)$, whereas the norm map $N_{K/\QQ}$ is of infinity type $(-1,-1)$.
Thus, the Hecke character $\overline{\psi^k} N_{K/\QQ}^{-k}$ is of infinity type $(k,0)$.

Henceforth, we fix a prime $v\mid\fp$ of $F$ and an embedding $\overline{\QQ}\subset \overline{\Qp}$ so that $v$ is sent into the maximal ideal $\MQp$ of $\cO_{\overline{\Qp}}$.
This allows us to consider $L_\fh^{\alg}(\overline{\varphi\psi^k})$ as elements of $\overline{\Qp}$.
Throughout, $\pi$ is a fixed uniformizer of $F_v$ and we write $\ord_\pi$ for the valuation on $\overline{\Qp}$ normalized so that $\ord_\pi(\pi)=1$.

\begin{Th}[Hida]
\label{thm:hida}
For all but finitely many characters $\varphi$ that factor through $\sR(\fg q^\infty)^\ac$, we have
\[
\ord_\pi\left(L_{(q)}^{\alg}(\overline{\varphi\psi^k})\right)=0.
\]
\end{Th}
\begin{proof}
For each $\varphi$, we have $\overline{\varphi}=\phi\eta$, where $\phi$ is a character of $\Delta$ and $\eta$ is a character of the Galois group $\Gal(\sR(\fg q^\infty)^\ac/F)$.
We may further decompose $\phi$ into $\phi'\nu^{-1}$, where $\nu$ is a character of $\Gal(F/\sR(\fg))$ and $\phi'$ is a character of $\Gal(\sR(\fg)/K)$.
We have the field diagram:

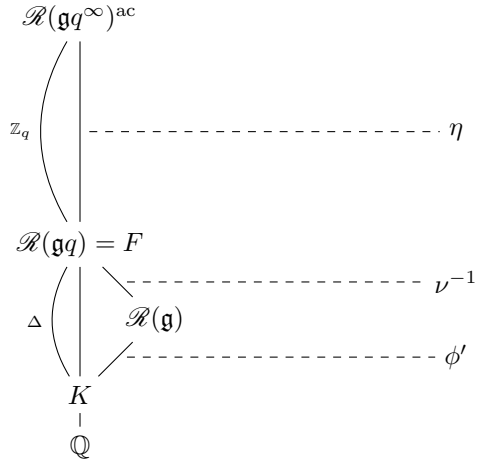
\begin{equation*}
\begin{tikzpicture}[node distance = 2cm, auto]
\node(Q) {$\QQ$};
\node(K) [above of =Q, node distance =0.7cm]{$K$};
\node (N1) [above of=K, right of =K, node distance =0.5cm] {\null};
\node (L) [right of =K, above of =K, node distance =1cm]{$\sR(\fg)$};
\node (N2) [above of=L, left of =L, node distance =0.5cm] {\null};
\node (F) [above of=K] {$\sR(\fg q)=F$};
\node (N3) [above of=F, node distance =1.5cm] {\null};
\node (I) [above of=F, node distance =3cm] {$\sR(\fg q^\infty)^\ac$};
\node (k) [right of =N3, node distance =5cm] {$\eta$};
\node (p) [below of =k, node distance =2cm] {$\nu^{-1}$};
\node (t) [below of =k, node distance =3cm] {$\phi'$};
\draw[-] (Q) to node {} (K);
\draw[dashed] (N1) to node {} (t);
\draw[dashed] (N2) to node {} (p);
\draw[-] (K) to node {} (L);
\draw[-] (K) to node {} (F);
\draw[bend left] (K) to node {\tiny$\Delta$} (F);
\draw[-] (L) to node {} (F);
\draw[bend left] (F) to node {\tiny$\Zq$} (I);
\draw[-] (F) to node {} (I);
\draw[dashed] (N3) to node {} (k);
\end{tikzpicture}
\end{equation*}
We take the CM field $M$ in \cite{hida2} to be the imaginary quadratic field $K$.
We take the CM type $\Sigma$ there to be the one that corresponds to the infinity type $(1,0)$ and $\kappa=0$.
Then the infinity type of the character $\lambda$ in \emph{op.~cit.} becomes
\[
k\Sigma+0(1-c)=k(1,0)+(0,0)-(0,0)=(k,0).
\]
The condition (M1) in \cite[Theorem~4.3]{hida2} does not hold since $K/\QQ$ is not unramified everywhere (it ramifies at the primes dividing the discriminant of $K$, which is nontrivial).
Hence, we can apply the aforementioned theorem with $\lambda$ and $\chi^{-1}$ taken to be $\overline{\psi^k}N^{-k}\phi'$ and $\eta$, respectively.
\end{proof}

\begin{rem}[{\cite[proof of Theorem~3.1.9]{Lam-thesis}}]
\label{remark: to go from L to Lh}
Let $\fg$ be a fixed ideal as before.
Fix an ideal $\fh$ of $\cO_K$ which is coprime to $\fp$ and divisible by $\fg q$.
Recall that the $\fh$-imprimitive algebraic $L$-value of $\overline{\varphi\psi^k}$ is given by
\[
L_{\fh}^{\alg}(\overline{\varphi \psi^k}) = \frac{L_{\fh}(\overline{\varphi\psi^k},k)}{\Omega_{\infty}^k}.
\]
Then, for almost all characters of $\Gal\left(\sR(\fg q^\infty)^{\ac}/F\right)\cong \Zq$, we have that
\[
\ord_{\pi}\left(L_{(q)}^{\alg}(\overline{\varphi \psi^k}) \right) = \ord_{\pi}\left(L_{\fh}^{\alg}(\overline{\varphi \psi^k}) \right).
\]
This follows from the observation that for a given prime ideal $\fa$ of $K$ that is coprime to $q$, for almost all characters $\eta$,
\[
\ord_{\pi}\left(1 - \frac{\overline{\varphi\psi^k}(\fa)}{(N\fa)^k} \right) = 0
\]
since $\eta$ sends $\fa$ to a $q$-power root of unity, and the images of $q$-power roots of unity under the reduction map on $\cO_{\overline{\Qp}}$ modulo $\MQp$ are distinct.
\end{rem}

\section{Consequences on class groups}
We now use Theorem~\ref{thm:hida} to study the growth of the $p$-part of the class group in an anticyclotomic $\Zq$-extension.
Let us introduce the necessary notation.
Throughout, $p\nmid 6 q$ is a fixed prime that is split in $K$ and $E_{/K}$ is a fixed CM elliptic curve as in the previous section (with Hecke character $\psi$ whose conductor is $\ff$). 
Let $K_0$ be any finite abelian extension of $K$ such that $p$ is unramified in $K_0$ and $p\nmid [K_0:K]$ (in the next subsection, we will let $K_0$ vary inside the anticyclotomic tower $\sR(\fg q^\infty)^\ac$).
Fix a prime $\fp$ of $K$ lying above $p$.
Set $L=K_0(E_\fp)$ and $L_\infty=L(E_{\fp^\infty})$.
Let $\Delta=\Gal(L/K)$ and $\Gamma=\Gal(L_\infty/L)\simeq \Zp$.
Let $\cG=\Gal(L_\infty/K)\cong \Delta\times\Gamma$ and $\Lambda=\Zp \llbracket \cG\rrbracket$.

Following \cite{Rub91}, we write $\overline\cC(L_\infty)$ (resp. $U(L_\infty)$) for the inverse limits over all finite sub-extensions inside $L_\infty$ of the completion of the elliptic units (resp. local principal units) at $\fp$.

Fix an ideal $\fh$ of $\cO_K$ which is coprime to $\fp$, is divisible by $\ff$, and is such that $K_0\subset K(E_{\fh})=\sR(\fh)$.
Let $\mu_K$ be the group of roots of unity of $K$ and $\lambda\in \cO_K\setminus\mu_K$ such that $\lambda\equiv 1\mod \fh$ with $(\lambda,6\fh\fp)=1$.
We let $\sigma_{(\lambda)}\in\Gal(K_0/K)$ denote the Artin symbol associated to $\lambda$.

We further decompose $\Delta$ as $H\times I$, where $H=\Gal(K_0/K)$ and $I=\Gal(K_0(E_\fp)/K_0)$.
Here, $I$ is the inertia subgroup at $\fp$ inside $\Delta$.
Let $\theta_\fp$ denote the canonical character given by the Galois action on $E_{\fp^\infty}$ restricted to $I$.
Given a character $\chi$ of $\Delta$, we write it as $\varphi\theta_\fp^k$, where $\varphi$ is a character of $H$ and $1\le k\le p-1$.
We have the following diagram:

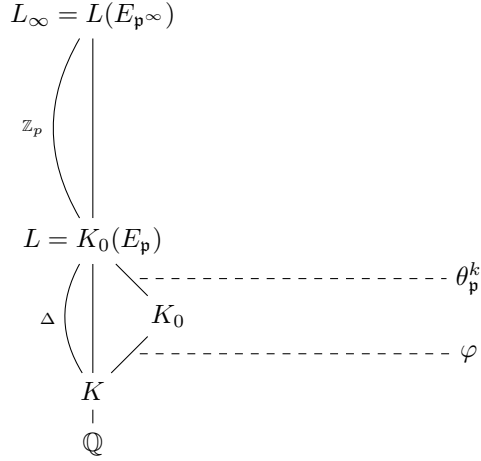
\begin{equation*}
\begin{tikzpicture}[node distance = 2cm, auto]
\node(Q) {$\QQ$};
\node(K) [above of =Q, node distance =0.7cm]{$K$};
\node (N1) [above of=K, right of =K, node distance =0.5cm] {\null};
\node (L) [right of =K, above of =K, node distance =1cm]{$K_0$};
\node (N2) [above of=L, left of =L, node distance =0.5cm] {\null};
\node (F) [above of=K] {$L=K_0(E_\fp)$};
\node (N3) [above of=F, node distance =1.5cm] {\null};
\node (I) [above of=F, node distance =3cm] {$L_\infty=L(E_{\fp^\infty})$};
\node (p) [below of =k, node distance =2cm] {$\theta_\fp^k$};
\node (t) [below of =k, node distance =3cm] {$\varphi$};
\draw[-] (Q) to node {} (K);
\draw[dashed] (N1) to node {} (t);
\draw[dashed] (N2) to node {} (p);
\draw[-] (K) to node {} (L);
\draw[-] (K) to node {} (F);
\draw[bend left] (K) to node {\tiny$\Delta$} (F);
\draw[-] (L) to node {} (F);
\draw[bend left] (F) to node {\tiny$\Zp$} (I);
\draw[-] (F) to node {} (I);
\end{tikzpicture}
\end{equation*}

Before proceeding, we need to introduce the notion of an anomalous prime.

\begin{Defi}
Fix a prime $p$ and a number field $\mathcal{K}$.
Let $v$ be a prime above $p$ in $\mathcal{K}$ and write $\kappa_v$ to denote the corresponding residue field.
In the sense of Mazur (see \cite[p.~186]{Maz72}), $E$ is \emph{anomalous} at $v$ if $\widetilde{E}(\kappa_v)[p]\neq 0$.
\end{Defi}

Let $w$ be a prime in $L_\infty$ which lies above $v$.
Denote by $\cZ$ the decomposition subgroup at $\fp$ inside $\cG$.
Since $\gcd(p,\abs{\Delta}) =1$, the action of $\Delta\cap \mathcal{Z}$ on $\mu_{p^\infty}(L_{\infty,w}) = \mu_{p^M}$ gives a $\Zp$-valued character which we denote by $\chi_{\cyc}: \Delta\cap \mathcal{Z}\rightarrow \mu_{p-1}\subseteq \Zp^\times$.

We now record a theorem of Lamplugh which will be important for our discussion.

\begin{Th}
\label{th:L-values-trivial}
Let $\chi=\varphi\theta_\fp^k$ be a character of $\Delta$.
When $E/K_0$ is anomalous at a prime above $\fp$, we assume that $\chi|_{\Delta\cap \cZ}$ is not the cyclotomic character.
Let $\fh$ and $\lambda$ be as above.
If
\[
\ord_\pi\left( \left(N\left(\lambda\right)-\lambda^k\varphi(\sigma_{\left(\lambda\right)})\right) \cdot L_{\fh}^{\alg}(\overline{\varphi\psi^k})\right)=0,
\]
then $\overline{\cC}(L_\infty)^\chi=U(L_\infty)^\chi$.
Here, $M^\chi$ denotes the $\chi$-isotypic component of a $\Lambda$-module $M$.
\end{Th}

\begin{proof}
See \cite[Theorem~7.7]{Lam15}.
\end{proof}

\subsection{Variations of class groups}
Let $F= \sR(\fg q)$ for some ideal $\fg$ of $\cO_K$ such that $\fg$ is divisible by $\ff$, is a product of primes that split in $K$, $p$ is unramified in $F/K$, $p\nmid [F:K]$, and is coprime to $\fp q$.
Furthermore, we assume that both $\fq$ and $\overline\fq$ are tamely ramified in $F$.
Then $\sR(\fg q^\infty)^{\ac}$ is a $\Zq$-extension of $F$, and for integers $n\ge0$, let $F_{n}/F$ be the $n$-th layer of this $\Zq$-extension.
Note that only primes above $q$ ramify in $F_{n}/F$, $p\nmid [F_{n}:K]$ (since $q\ne p$), and $F_n \subseteq \sR(\fg q^{n+1})$.
Therefore, we may take $K_0$ and $\fh$ in the previous section to be $F_n$ and $\fg q^{n+1}$, respectively.

For $n\ge1$, just as before we define $L_{n}=F_{n}(E_\fp)$, $L_{n,\infty}=F_{n}(E_{\fp^\infty})$, $\Delta_{n}=H_{n}\times I$, $\cG_{n}=\Delta_{n}\times \Gamma$, $U_{n,\infty}=U(L_{n,\infty})$, etc.
As mentioned previously, we now let $K_0$ vary inside the anticyclotomic tower $\sR(\fg q^\infty)^\ac$.
Note that $I=\Gal(K_0(E_\p)/K_0)\cong \Gal(L_n/F_n)$.
Define $X_{n,\infty}$ to be the Galois group of the maximal abelian $p$-extension of $L_{n,\infty}$ which is unramified outside $\fp$.
By global class field theory we have the following four-term exact sequence
\begin{equation}
\label{eqn: four term}
0 \rightarrow \overline{\mathcal{E}}_{n,\infty}/\overline{\mathcal{C}}_{n,\infty} \rightarrow U_{n,\infty}/\overline{\mathcal{C}}_{n,\infty} \rightarrow X_{n,\infty} \rightarrow A_{n,\infty} \rightarrow 0.
\end{equation}
Here, $\overline{\mathcal{E}}_{n,\infty}= \overline{\mathcal{E}}(L_{n,\infty})$ is used to denote the global units of $L_{n,\infty}$.
Finally, $A_{n,\infty} = A(L_{n,\infty})$ is the inverse limit of the $p$-part of the class group for each finite extension of $F_n$ contained inside $L_{n,\infty}$.
In other words, $A_{n,\infty}$ can be identified with the Galois group of the maximal abelian unramified $p$-extension of $L_{n,\infty}$.
We now prove the key result which will be required in proving Theorem~\ref{thmA}.

\begin{Th}
\label{key result for class group}
There exists an integer $N\ge0$ such that $X_{n,\infty}^I=X_{N,\infty}^I$ for all $n\ge N$, where $M^I$ denotes the subgroup of $M$ fixed by $I$.
\end{Th}

\begin{proof}
Let $n\ge0$ be an integer and consider a character $\eta$ of $\Gal(F_n/F)\cong \ZZ/q^n$.
Let $\phi_0$ be a character of $\Gal(F/K)$ and $k$ an integer that is a multiple of $p-1$ so that $\theta_\fp^{k}$ is the trivial character.
Set $\varphi = \eta\phi_0$.
We draw the following field diagram for the convenience of the reader.
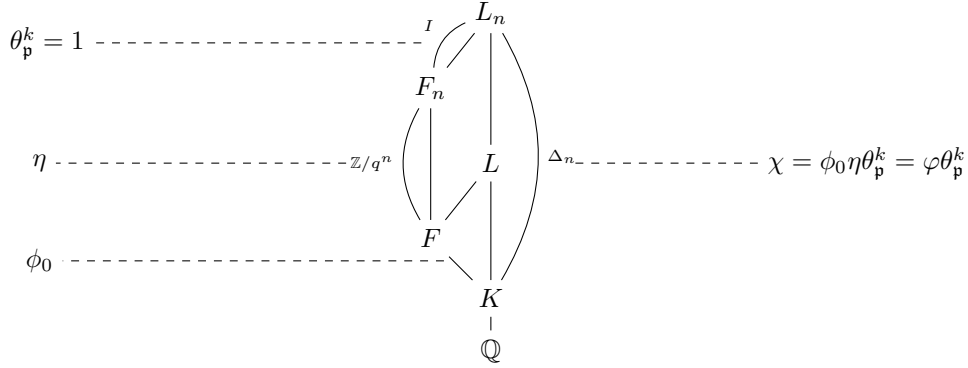
\begin{equation*}
\begin{tikzpicture}[node distance = 2cm, auto]
\node(Q) {$\QQ$};
\node(K) [above of =Q, node distance =0.7cm]{$K$};
\node (F) [above of=K, left of=K, node distance=0.8cm] {$F$};
\node (L) [above of=K, node distance =1.8cm] {$L$};
\node (N1) [above of=K, left of =K, node distance =0.5cm] {\null};
\node (p0) [left of=N1, node distance =5.5cm] {$\phi_0$};
\node (N3) [left of=L, node distance =1.8cm] {\null};
\node (N4) [right of=L, node distance =1cm] {\null};
\node (Fn) [above of=F, node distance =2cm] {$F_n$};
\node (N2) [above of=Fn, node distance =0.6cm] {\null};
\node (Ln) [above of=L, node distance =2cm] {$L_n$};
\node (e) [left of =N3, node distance =4.2cm] {$\eta$};
\node (t) [left of =N2, node distance =5.1cm] {$\theta_{\fp}^k=1$};
\node (vp) [right of =L, node distance =5cm] {$\chi = \phi_0\eta\theta_{\fp}^k = \varphi\theta_{\fp}^k$};
\draw[-] (Q) to node {} (K);
\draw[-] (L) to node {} (F);
\draw[-] (L) to node {} (K);
\draw[-] (Fn) to node {} (F);
\draw[-] (Fn) to node {} (Ln);
\draw[-] (L) to node {} (Ln);
\draw[-] (K) to node {} (F);
\draw[bend right] (K) to node [swap]{\tiny$\Delta_n$} (Ln);
\draw[bend right] (Ln) to node [swap]{\tiny$I$} (Fn);
\draw[bend left] (F) to node {\tiny$\ZZ/q^n$} (Fn);
\draw[dashed] (N1) to node {} (p0);
\draw[dashed] (N2) to node {} (t);
\draw[dashed] (N3) to node {} (e);
\draw[dashed] (N4) to node {} (vp);
\end{tikzpicture}
\end{equation*}

Let $\cO$ denote the ring of integers of the unique unramified $\Zq$-extension of $F_v$.
In other words, $\cO = \cO_{F_v(\mu_{q^{\infty}})}$.
Let $\lambda\in \cO_K\setminus \mu_K$ such that $\lambda\equiv 1\mod \fh$ and $(\lambda,6\fh\fp)=1$ (where $\fh=\fg q^{n+1}$). 
We have
\begin{align*} \left(\lambda\overline{\lambda}-\lambda^k\varphi(\sigma_{\left(\lambda\right)})\right) \equiv 0 \pmod{\pi \cO} &\Leftrightarrow \varphi(\sigma_{(\lambda)}) \equiv \overline{\lambda}\lambda^{1-k} \pmod{\pi \cO}\\
&\Leftrightarrow \eta\phi_0(\sigma_{(\lambda)}) \equiv \overline{\lambda}\lambda^{1-k} \pmod{\pi \cO}\\
&\Leftrightarrow \eta(\sigma_{(\lambda)}) \equiv \overline{\lambda}\lambda^{1-k}\phi_0^{-1}(\sigma_{(\lambda)}) \pmod{\pi \cO}.
\end{align*}
Note that $\eta$ has exact order $q^m$ for some $m\geq 1$.
Therefore, $\eta(\sigma_{(\lambda)})$ is a primitive $q^m$-th root of unity.
But in $\cO/\pi\cO$, the $q$-power roots of unity are distinct.
Therefore, by the same argument outlined in Remark~\ref{remark: to go from L to Lh}, there exists an integer $N_1$ such that for all characters $\eta$ of $\Gal(F_n/F)$  which do not factor through $\Delta_{N_1}$ (with $n\ge N_1$),
\begin{equation}
\ord_{\pi}\left(N_{K/\QQ}\left(\lambda\right)-\lambda^k\varphi(\sigma_{\left(\lambda\right)})\right) =0.
\label{eq:Euler-factor}    
\end{equation}

By Theorem~\ref{thm:hida} and Remark~\ref{remark: to go from L to Lh}, one can choose a sufficiently large $N_2$ such that
\begin{equation}
\label{eq:ord-L} \ord_{\pi}\left(L_{\fh}^{\alg}(\overline{\varphi\psi^k}) \right)=0   
\end{equation}
 for all characters $\varphi$ of $\Delta_n$ which do not factor through $\Delta_{N_2}$ (with $n\ge N_2$).

Set $N=\max(N_1,N_2)$. If  $\chi=\varphi$ is a character of $\Delta_n$ which does not factor through $\Delta_N$ (with $n\ge N$), then  \eqref{eq:Euler-factor} and \eqref{eq:ord-L} imply that
\[
\ord_\pi\left( \left(N_{K/\QQ}\left(\lambda\right)-\lambda^k\varphi(\sigma_{\left(\lambda\right)})\right) \cdot L_{\fh}^{\alg}(\overline{\varphi\psi^k})\right)=0.
\]
Take $K_0$ in Theorem~\ref{th:L-values-trivial} to be $F_n$. Since the restriction of the character $\varphi$ to $I$ is trivial, the hypothesis regarding $E/K_0$ when a prime above $\fp$ is an anomalous prime always holds.
Therefore, we deduce that
\[
U_{n,\infty}^\chi=\overline{\cC}_{n,\infty}^\chi
\]
for all characters $\chi$ of $\Delta_{n}$ that do not factor through $\Delta_N$ with $\chi|_I=1$.
This implies
\[
U_{n,\infty}^I/\overline{\cC}_{n,\infty}^I=U_{N,\infty}^I/\overline{\cC}_{N,\infty}^I.
\]
Next, via the main conjecture of Iwasawa theory for imaginary quadratic fields (see \cite[Theorem~4.1(i)]{Rub91}) we can conclude that there exists an integer $N\geq 0$ such that
\[
\charac_{\Lambda}(X_{n,\infty}^I) = \charac_{\Lambda}(X_{N,\infty}^I)
\]
for all $n\geq N$.
Now, consider the restriction map
\[
\pi_{n,N}: X_{n,\infty}^I \twoheadrightarrow X_{N,\infty}^I.
\]
Since characteristic ideals are multiplicative in short exact sequences, the kernel of the above surjective map must be finite. 
However, a theorem of R.~Greenberg (see \cite[Theorem \S1]{Gre78}) ensures that there are no non-trivial finite submodules inside $X_{n,\infty}^I$.
This forces the kernel to be trivial, i.e.,
\[
X_{n,\infty}^I = X_{N,\infty}^I.
\]
The proof of the theorem is now complete.
\end{proof}

We can now state and prove the auxiliary result that will allow us to conclude Theorem~\ref{thmA}.
Our proof follows the proof of \cite[Theorem~7.10]{Lam15} very closely.
We repeat the statement below for the convenience of the reader.

\begin{Th}
\label{auxiliary}
Let $K$ be an imaginary quadratic field of class number 1.
Let $p$ and $q$ be distinct primes ($\geq 5$) which split in $K$.
Let $\fg$ be a fixed ideal of $\cO_K$ coprime to $pq$ such that $\fg$ is a product of split primes and is divisible by the conductor of an elliptic curve over $K$ with CM by $\cO_K$.
Let $F= \sR(\fr q)$.
We assume that $p\nmid[F:K]$.
Let $\sR(\fr q^\infty)^{\ac}/F$ denote the anticyclotomic $\Zq$-extension and write $F_n$ for the unique subextension of $\sR(\fr q^\infty)^{\ac}/F$ whose degree is $q^n$.
Then, there exists an integer $N$ such that for all $n\ge N$,
\[\ord_p(h(F_n))=\ord_p(h(F_N)).\]
\end{Th}

\begin{proof}
Let the $p$-class group of $F_n$ (resp. $F_N$) be denoted by $A(F_n)$ (resp. $A(F_N)$).
Since $p$ does not divide $[F_n : F_N]$, we have an injection
\begin{equation}
 \label{injection}
A(F_N) \hookrightarrow A(F_n).
\end{equation}
It follows from global class field theory that for all $n\geq 0$, we have the identification
\[
A_{n,\infty} \simeq \Gal(M_{n,\infty}/L_{n,\infty}),
\]
where $M_{n,\infty}$ is the maximal abelian unramified $p$-extension of $L_{n,\infty}$.
Consider the following diagram\[
\xymatrix{
 A_{N,\infty}^I\ar[r]\ar[d] & A_{n,\infty}^I\ar[d]\\ 
 A(F_N)\ar[r]& A(F_n)}
\]
where the vertical maps are given by restriction and are surjective because the extension $L_{n,\infty}/F_n$ and $L_{N,\infty}/F_N$ are totally ramified at primes above $\fp$.
Furthermore, the top horizontal map is surjective by Theorem~\ref{key result for class group} and the exact sequence \eqref{eqn: four term}.
Therefore, the bottom row is a surjective map as well.
When combined with \eqref{injection}, we see that the bottom row is in fact an isomorphism.
This completes the proof of the theorem.
\end{proof}

The following lemma allows us to complete the proof of Theorem~\ref{thmA} via Theorem~\ref{auxiliary}.

\begin{Lemma}
\label{lemma: LCM}
Let $\fa$ and $\fb$ be ideals of $\cO_K$.
If $p\nmid [\sR(\fa):K]\cdot [\sR(\fb):K]$, then $p\nmid [\sR(\lcm(\fa,\fb)):K]$.
\end{Lemma}

\begin{proof}
Let us write $\fa=\prod\fp_i^{m_i}$, $\fb=\prod\fp_i^{n_i}$, where $\fp_i$ are distinct prime ideals of $\cO_K$.
Recall that $K$ is of class number 1. By the theory of complex multiplication, if $I$ is an ideal of $\cO_K$, we have 
\[
\Gal(\sR(I)/K)=\Gal(K(E[I])/K)\cong (\cO_K/I)^\times.
\]
Thus, by the Chinese remainder theorem, 
\[
p\nmid \abs{(\cO_K/\fp_i^{m_i})^\times},\quad p\nmid \abs{(\cO_K/\fp_i^{n_i})^\times}
\]
for all $i$. As $\lcm(\fa,\fb)=\prod \fp_i^{\max(m_i,n_i)}$, we deduce that
\[
p\nmid \abs{(\cO_K/\lcm(\fa,\fb))^\times}=[\sR(\lcm(\fa,\fb)):K]. \qedhere
\]
\end{proof}

We can now prove Theorem~\ref{thmA} from the introduction.

\begin{Theorem*}
Let $K$ be an imaginary quadratic field of class number 1.
Let $p$ and $q$ be distinct primes ($\geq 5$) which split in $K$.
Let $\fr$ be a fixed ideal of $\cO_K$ coprime to $pq$ such that $\fr$ is a product of split primes.
Let $\cF= \sR(\fr q)$ and write $\sR(\fr q^\infty)^{\ac}/\cF$ for the anticyclotomic $\Zq$-extension.
Assume that $p\nmid[\cF:K]$.
Then, there exists an integer $N$ such that for all $n\ge N$,
\[
\ord_p(h(\cF_n))=\ord_p(h(\cF_N)).
\]
\end{Theorem*}

\begin{proof}
Let ${E}_{/K}$ be a CM elliptic curve of conductor $\ff$ such that all the prime divisors of $\ff$ are split in $K$ but the prime divisors (which are $\geq 5$) of $[\sR(\ff):K]$ are not split in $K$.
Such elliptic curves exist as we have seen in \S\ref{S:CM}.

Let $\fr$ be any ideal of $\cO_K$ and $p,q$ be two distinct primes satisfying the hypotheses in the statement of the theorem.
Set $\fg = \lcm(\ff,\fr)$ and define $F = \sR(\fg q)$.
By assumption, $p\nmid [\sR(\fr q):K]$ and we have chosen our auxiliary CM elliptic curve so that $p\nmid [\sR(\ff):K]$.
Thus, it follows from Lemma~\ref{lemma: LCM} that $p\nmid [\sR(\fg q):K]$.  Furthermore, both $\ff$ and $\fr$ are only divisible by split primes. 
Therefore, Theorem~\ref{auxiliary} holds for the ideal $\fg$.

Since $p\nmid [F_n: \cF_n]$ and $p\nmid[\cF_{n+1}:\cF_n]$ for all $n\geq 0$, we have
\[
A(\cF_n) \hookrightarrow A(F_n), \quad A(\cF_n) \hookrightarrow A(\cF_{n+1}).
\]
Theorem~\ref{auxiliary} asserts that $\ord_p(h(F_n))$ stabilizes as $n\rightarrow \infty$.
Hence, the same is true for $\ord_p(h(\cF_n))$.
\end{proof}

\section{Asymptotic growth of fine Selmer groups of abelian varieties}

\subsection{Definition of fine Selmer groups}\label{sec:notation}
Suppose $F$ is a number field.
Throughout, $A_{/F}$ is a fixed abelian variety.
We fix a finite set $S$ of primes of $F$ containing $p$, the primes dividing the conductor of $A$, as well as the Archimedean primes.
We write $S_f$ to denote the set of finite primes.
Denote by $F_S$, the maximal algebraic extension of $F$ unramified outside $S$.
For every (possibly infinite) extension $L$ of $F$ contained in $F_S$, write $G_S\left({L}\right) = \Gal\left(F_S/{L}\right)$.
Write $S\left(L\right)$ for the set of primes of $L$ above $S$.
If $L$ is a finite extension of $F$ and $w$ is a place of $L$, we write $L_w$ for its completion at $w$; when $L/F$ is infinite, it is the union of completions of all finite sub-extensions of $L$.

\begin{Defi}
Let $L/F$ be an algebraic extension.
The \emph{$p$-primary fine Selmer group of $A$} over $L$ is defined as
\[
\Sel_0(A/L)=\ker\left( H^1\left(G_S\left(L\right), A[p^\infty]\right) \rightarrow \bigoplus_{v\in S} H^1\left(L_v, A[p^\infty]\right)\right).
\]
Similarly, the \emph{$p$-fine Selmer group of $A$} over $L$ is defined as
\[
\Sel_0(A[p]/L)=\ker\left( H^1\left(G_S\left(L\right), A[p]\right) \rightarrow \bigoplus_{v\in S} H^1\left(L_v, A[p]\right)\right).
\]
\end{Defi}

Note that $\Sel_0(A/L)$ is independent of the choice of $S$, whereas the definition of $\Sel_0(A[p]/L)$ depends on $S$; see for example \cite[Lemma~4.1 and p.~86]{LM16}.
Since our main result concerns $\Sel_0(A/L)$, we suppress $S$ from the notation of $\Sel_0(A[p]/L)$ for simplicity.

It is easy to observe that if $F_\infty/F$ is an infinite extension,
\[
\Sel_0\left( A/F_\infty\right) = \varinjlim_L \Sel_0\left(A/L \right), \quad \Sel_0\left( A[p]/F_\infty\right) = \varinjlim_L \Sel_0\left(A[p]/L \right),
\]
where the inductive limits are taken with respect to the restriction maps and $L$ runs over all finite extensions of $F$ contained in $F_\infty$.
Next, we define the notion of $p$-rank of an abelian group $G$.
\begin{Defi}
Let $G$ be an abelian group.
Define the \emph{$p$-rank} of $G$ as
\[
r_p (G) = r_p (G[p]) := \dim_{\mathbb{F}_p}\left(G[p]\right).
\] 
\end{Defi}

\subsection{Growth of fine Selmer groups in \texorpdfstring{$\Zq$}{}-extensions}
In this section, we prove the following theorem which essentially says that the $p$-part of the class group and the $p$-primary fine Selmer group have similar growth behaviour in $\Zq$-extensions.
Our result is motivated by \cite[Section~5]{LM16}.

\begin{Th}
\label{class group and fine Selmer group}
Let $A$ be a $d$-dimensional abelian variety defined over a number field $F$.
Let $S(F)$ be a finite set of primes in $F$ consisting precisely of the primes above $q$, the primes of bad reduction of $A$, and the Archimedean primes.
Let $F_\infty/F$ be a fixed $\Zq$ extension such that primes in $S_{f}(F)$ are finitely decomposed in $F_\infty/F$ and suppose $[F_n:F] = q^n$.
Further suppose that $A[p]\subseteq A(F)$.
Then as $n \rightarrow\infty$,
\[
\abs{r_{p}\left(\Sel_0\left(A/F_n\right)\right) - 2d r_{p}\left(\Cl(F_n)\right)} =O(1).
\]
\end{Th}

If $A[p]\subseteq A(F)$, then the action of $G_F$ on $A[p]$ is trivial.
Let $A^{\vee}$ be the dual abelian variety.
The action on the dual representation, $A^{\vee}[p]$ is also trivial.
This tells us that $A^{\vee}[p] \subseteq A^{\vee}(F)$.
Therefore, Theorem~\ref{class group and fine Selmer group} allows us to deduce the following result.
\begin{Cor}
With the same hypothesis as in Theorem~\ref{class group and fine Selmer group}
\[
\abs{r_{p}\left(\Sel_0\left(A/F_n\right)\right) - r_{p}\left(\Sel_0\left(A^{\vee}/F_n\right)\right)} =O(1).
\]
\end{Cor}

To prove Theorem~\ref{class group and fine Selmer group}, we need a few lemmas.

\begin{Lemma}
\label{lemma 3.2 LM16}
Consider the following short exact sequence of co-finitely generated abelian groups
\[
P \rightarrow Q \rightarrow R \rightarrow S.
\]
Then,
\[
\abs{ r_{p}\left( Q\right) - r_{p}\left( R\right)} \leq 2r_{p}\left(P\right) + r_{p} \left(S\right).
\]
\end{Lemma}

\begin{proof}
See \cite[Lemma~3.2]{LM16}.
\end{proof}

\begin{Lemma}
\label{lemma relate S class group to the regular one}
Let $F_\infty$ be any $\Zq$-extension of $F$ such that all the primes in $S_{f}(F)$ are finitely decomposed.
Let $F_n$ be the subfield of $F_\infty$ such that $[F_n: F] = q^{n}$.
Then
\[
\abs{r_{p}\left(\Cl(F_n)\right) - r_{p}\left(\Cl_S(F_n)\right)} = O(1).
\]
\end{Lemma}

\begin{proof}
For each $F_n$, we write $S_f(F_n)$ for the set of finite primes of $F_n$ above $S_f$.
For each $n$, we have the following exact sequence 
\[
\ZZ^{\abs{S_f(F_n)}} \longrightarrow\Cl(F_n) \xrightarrow{\alpha_n} \Cl_S(F_n) \longrightarrow 0
\] (see \cite[Lemma~10.3.12]{NSW08}).
Since the class group is always finite, it follows that $\ker(\alpha_n)$ is finite.
Also, $r_{p}\left(\ker(\alpha_n)\right) \leq \abs{S_f(F_n)}$ and $r_{p}\left(\ker(\alpha_n)/p\right)\leq \abs{S_f(F_n)}$.
By Lemma~\ref{lemma 3.2 LM16},
\[
\abs{r_{p}\left( \Cl(F_n)\right) - r_{p}\left( \Cl_S(F_n)\right)} \leq 2\abs{S_f(F_n)} = O(1). \qedhere
\]
\end{proof}

\begin{Lemma}\label{lem:control-p}
Let $F_\infty/F$ be a $\Zq$-extension and let $F_n$ be the subfield of $F_\infty$ such that $[F_n:F]=q^n$.
Let $A$ be an abelian variety defined over $F$.
Suppose that all primes of $S_{f}(F)$ are finitely decomposed in $F_\infty/F$.
Then
\[
\label{eqn: fine Selmer ell-fine Selmer}
\abs{ r_{p} \left( \Sel_0(A[p]/F_n)\right) - r_{p} \left( \Sel_0(A/F_n)\right)} =O(1).
\]
\end{Lemma}

\begin{proof}
Consider the commutative diagram
\begin{align*}
\begin{matrix}
0&\rightarrow &\Sel_0(A[p]/F_n)&\rightarrow&H^1(G_S(F_n), A[p])&\rightarrow& \bigoplus_{v\in S(F_n)} H^1(F_{n,v_n}, A[p])
\cr \hbox{ } &&\Bigg\downarrow s_n &&\Bigg\downarrow f_n &&\Bigg\downarrow \gamma_n
\cr 0&\rightarrow &\Sel_0(A/F_n)[p]&\rightarrow &H^1(G_S(F_n), \ A[p^\infty])[p]&\rightarrow& \bigoplus_{v_n \in S(F_n)} H^1(F_{n,v_n}, \ A[p^\infty])[p] .
\end{matrix}
\end{align*}
Both $f_n$ and $\gamma_n$ are surjective.
Since $A[p]\subset A(F_n)$, the kernel of these maps are given by
\begin{align*}
\ker(f_n) & = A(F_n)[p^\infty]\big/{p}  \simeq \left( \ZZ/p\ZZ\right)^{2d},\\
\ker(\gamma_n) & = \bigoplus_{v_n\in S(F_n)} A(F_{n,v_n})[p^\infty]\big/p \simeq \bigoplus_{v_n\in S_f(F_n) }\left( \ZZ/p\ZZ\right)^{2d},
\end{align*}
where the last isomorphism follows from our assumption that $p$ is odd.

Observe that $r_{p} \left( \ker\left(s_n\right)\right) \leq r_{p} \left( \ker\left(f_n\right)\right) = 2d$ and that $r_{p} \left( \ker\left(\gamma_n \right)\right) = 2d \abs{ S_f(F_n)}$.
By hypothesis, $S_f(F_n)$ is bounded as $n$ varies.
It follows from the snake lemma that both $r_{p} \left( \ker\left( s_n \right)\right)$ and $r_{p} \left( \coker\left( s_n \right)\right)$ are finite and bounded.
Applying Lemma~\ref{lemma 3.2 LM16} to the following exact sequence
\[
0 \rightarrow \ker(s_n) \rightarrow \Sel_0(A[p]/F_n) \rightarrow \Sel_0(A/F_n)[p]\rightarrow \coker (s_n) \rightarrow 0
\]
completes the proof.
\end{proof}

\begin{proof}[Proof of Theorem~\ref{class group and fine Selmer group}]
By hypothesis, $A[p]\subseteq A(F)$.
Therefore, $A[p]\simeq (\ZZ/p)^{2d}$.
We have
\[
H^1\left( G_S(F_n), \ A[p]\right) = \Hom\left( G_S(F_n), \ A[p]\right).
\]
There are similar identifications for the local cohomology groups.
Thus,
\[
\Sel_{0}\left( A[p]/F_n\right) \ \simeq \Hom\left( \Cl_S(F_n), A[p]\right) \simeq \Cl_S(F_n)[p]^{2d}
\]
as abelian groups.
Therefore,
\[
r_{p}\left( \Sel_0\left( A[p]/F_n\right)\right) = 2d r_{p}\left( \Cl_S(F_n)\right).
\]
The theorem now follows from Lemmas~\ref{lemma relate S class group to the regular one} and~\ref{lem:control-p}.
\end{proof}

Let $p^{e_n}$ be the largest power of $p$ that divides the class number of $F_n$.
If $e_n$ is bounded then it follows (trivially) that the $p$-rank is bounded.
Thus, the following corollary is immediate.

\begin{Cor}
Let $p\neq q$.
Let $F/\QQ$ be any finite extension of $\QQ$ and $F_\infty/F$ be any $\Zq$-extension of $F$.
Let $p^{e_n}$ be the exact power of $p$ dividing the class number of the $n$-th intermediate field $F_n$.
Let $A_{/F}$ be an abelian variety such that $A[p]\subseteq A(F)$.
If $e_n$ is bounded as $n\rightarrow \infty$, then $r_{p} \left( \Sel_0\left( A/F_n\right)\right)$ is bounded independently of $n$.
\end{Cor}

In addition to Theorem~\ref{thmA}, there are some other results in the literature where it is known that the $p$-part of the class group stabilizes in a $\Zq$-extension (when $p,q$ are distinct primes).
These were discussed briefly in the introduction and are recorded here more precisely.
\begin{enumerate}
\item (\cite[Theorem]{Was78})
Let $F/\QQ$ be an abelian extension of $\QQ$ and $F_\infty/F$ be the cyclotomic $\Zq$-extension of $F$.
If $p^{e_n}$ be the exact power of $p$ dividing the class number of the $n$-th intermediate field $F_n$, then $e_n$ is bounded as $n\rightarrow \infty$.
\item (\cite[Theorem~7.10]{Lam15}) 
Let $p,q$ be fixed odd distinct primes both $\geq 5$, $K$ be an imaginary quadratic field of class number 1 where $p$ and $q$ split, and $E_{/K}$ be an elliptic curves with CM by $\cO_K$ and good reduction at $p,q$.
Let $K_\infty$ be the $\Zq$ extensions of $K$ which is unramified outside $\fq$ (resp. $\overline{\fq}$).
Let $\fg$ be a fixed ideal of $\cO_K$ such that it is coprime to $pq$
and $F=\sR(\fg \fq)$ is of degree prime-to-$p$ over $F$.
Then, the $p$-part of the class number stabilizes in $FK_\infty=\sR(\fg \fq^\infty)$.
However, since $p$ is assumed to be unramified in $F$ in \emph{loc.~cit.}, the hypothesis $A[p]\subseteq A(F)$ in Theorem~\ref{class group and fine Selmer group} is unlikely to hold.
The same can be said regarding the setting studied in Theorem~\ref{thmA}.
\end{enumerate}

\begin{Th}
\label{whole FSG stabilizes}
With notation as above, suppose that the $p$-rank of the fine Selmer group, denoted by $r_p\left(\Sel_0(A/F_n) \right)$ stabilizes in a $\Zq$-extension of $F$.
Then there exists $n\ge 0$, such that for all $m\geq n$, the restriction map induces an isomorphism
\[
\Sel_{0}(A/F_n) \simeq \Sel_{0}(A/F_m).
\]
\end{Th}

\begin{proof}
The following argument is similar to the one presented in \cite[p.~15]{Lam-thesis}, where instead of classical Selmer groups, we consider fine Selmer groups.
Consider the extension $F_m/F_n$.
Then $[F_m : F_n] = q^{m-n} = t$ (say).
The restriction map
\[
\Gal\left( \overline{F}/F_n\right) \longrightarrow \Gal\left( \overline{F}/F_m\right)
\]
induces the restriction homomorphism
\[
\textrm{res}: \Sel_0(A/F_{n}) \longrightarrow \Sel_0(A/F_{m}).
\]
Since $\gcd(q, p)=1$, this maps is an injection.
Moreover, we have
\[
\Sel_0(A/F_{n}) \xrightarrow{\rm{res}} \Sel_0(A/F_{m}) \xrightarrow{\rm{cores}} \Sel_0(A/F_{n}) \xrightarrow{t^{-1}} \Sel_0(A/F_{n})
\]
where $\rm{cores} \circ \rm{res} = t$.
The composition $\rm{res}\circ \rm{cores} \circ t^{-1}$ is the identity map; thus, the exact sequence
\[
0 \longrightarrow \Sel_0(A/F_{n}) \longrightarrow \Sel_0(A/F_{m}) \longrightarrow \Sel_0(A/F_{m})\big/ \Sel_0(A/F_{n}) \longrightarrow 0
\]
is split exact.

Let us write $\Sel_0(A/F_n)=(\Qp/\Zp)^{s_n}\oplus T_n$, where $s_n\ge 0$ and $T_n$ is a finite $p$-group.
Then,
\[
r_p\left(\Sel_0(A/F_n)\right)=s_n+r_p(T_n).
\]
The injection $ \Sel_0(A/F_{n}) \hookrightarrow \Sel_0(A/F_{m}) $ tells us that $s_m\ge s_n$.
If the $p$-rank $r_p\left(\Sel_0(A/F_n)\right)$ eventually stabilizes it follows that $s_n$ also stabilizes.
Denote the cokernel of the injection by $C_{m,n}$.
By duality, we have the short exact sequence
\[
0\rightarrow C_{m,n}^\vee\rightarrow \Zp^{s_m}\oplus T_m^{\vee}\rightarrow \Zp^{s_n}\oplus T_n^{\vee}\rightarrow 0.
\]
When $s_m=s_n$, $C_{m,n}^\vee$ must be finite.
Consequently, the image of $C_{m,n}^\vee$ in $\Sel_0(A/F_n)^\vee$ is contained inside $T_m^{\vee}$.
Furthermore, since the short exact sequence splits, we deduce the isomorphism
\[T_m=T_n\oplus C_{m,n}.\]
As $s_n$ stabilizes, $r_p(T_n)$ also stabilizes.
Therefore, $C_{m,n}$ has to be $0$ eventually.
\end{proof}

Theorem~\ref{thmB} is now an immediate corollary of Theorems \ref{class group and fine Selmer group} and \ref{whole FSG stabilizes}.
\begin{Cor}
Let $p,q$ be distinct odd primes.
Let $F$ be any number field and $A_{/F}$ be an abelian variety such that $A[p]\subseteq A(F)$.
Let $F_\infty/F$ be a $\Zq$-extension where the primes above $q$ and the primes of bad reduction of $A$ are finitely decomposed.
If the $p$-part of the class group stabilizes, i.e., there exists $N\ge 0$ such that for all $n\geq N$,
\[
\ord_{p}(h(F_n)) = \ord_{p}(h(F_N)),
\]
then the growth of the $p$-primary fine Selmer group stabilizes in the $\Zq$-extension as well, i.e., there exists an integer $N'\ge N$ such that for all $n\geq N'$, the restriction map induces an isomorphism
\[
\Sel_0(A/F_n) \simeq\Sel_0(A/F_{N'}).
\]
\end{Cor}

\bibliographystyle{amsalpha}
\bibliography{references}
\end{document}